\documentclass[12pt]{amsart}
\usepackage{amssymb}
\usepackage[cmtip,all]{xy}
\usepackage{verbatim}
\usepackage{upref}
\usepackage[usenames,dvipsnames]{color}
\usepackage{url}
\usepackage[normalem]{ulem}

\allowdisplaybreaks

\newtheorem{thm}{Theorem}[section]
\newtheorem*{thm*}{Theorem}
\newtheorem{lem}[thm]{Lemma}
\newtheorem{cor}[thm]{Corollary}
\newtheorem{prop}[thm]{Proposition}

\theoremstyle{definition} % definition style
\newtheorem{defn}[thm]{Definition}
\newtheorem{ex}[thm]{Example}
\newtheorem{q}[thm]{Question}
\newtheorem{notn}[thm]{Notation}
\theoremstyle{remark} % remark style
\newtheorem{rem}[thm]{Remark}

\numberwithin{equation}{section}

\newcommand{\secref}[1]{Section~\textup{\ref{#1}}}

\newcommand{\corref}[1]{Corollary~\textup{\ref{#1}}}
\newcommand{\lemref}[1]{Lemma~\textup{\ref{#1}}}
\newcommand{\propref}[1]{Proposition~\textup{\ref{#1}}}

\newcommand{\defnref}[1]{Definition~\textup{\ref{#1}}}
\newcommand{\remref}[1]{Remark~\textup{\ref{#1}}}
\newcommand{\qref}[1]{Question~\textup{\ref{#1}}}
\newcommand{\exref}[1]{Example~\textup{\ref{#1}}}

\newcommand{\FF}{\mathcal F}

\newcommand{\RR}{\mathcal R}
\newcommand{\HH}{\mathcal H}

\newcommand{\N}{\mathbb N}
\newcommand{\Z}{\mathbb Z}

\newcommand{\T}{\mathbb T}
\newcommand{\R}{\mathbb R}

\renewcommand{\bar}{\overline}
\newcommand{\what}{\widehat}
\newcommand{\wilde}{\widetilde}
\newcommand{\inv}{^{-1}}
\newcommand{\<}{\langle}
\renewcommand{\>}{\rangle}

\newcommand{\ann}{^\perp}
\newcommand{\pann}{{}\ann}
\newcommand{\Chi}{\raisebox{2pt}{\ensuremath{\chi}}}
\renewcommand{\epsilon}{\varepsilon}

\DeclareMathOperator*{\spn}{span}
\DeclareMathOperator*{\clspn}{\overline{\spn}}

\newcommand{\righttext}[1]{\quad\text{#1 }}

\newcommand{\wkstcl}{^{\text{weak*}}}

\begin{document}
\title[Odered invariant ideals of $B(G)$]{Ordered invariant ideals of Fourier-Stieltjes algebras}
\author[Kaliszewski]{S. Kaliszewski}
\address{School of Mathematical and Statistical Sciences
\\Arizona State University
\\Tempe, Arizona 85287}
\email{kaliszewski@asu.edu}
\author[Landstad]{Magnus~B. Landstad}
\address{Department of Mathematical Sciences\\
Norwegian University of Science and Technology\\
NO-7491 Trondheim, Norway}
\email{magnusla@math.ntnu.no}
\author[Quigg]{John Quigg}
\address{School of Mathematical and Statistical Sciences
\\Arizona State University
\\Tempe, Arizona 85287}
\email{quigg@asu.edu}
%\date{Revised draft, \today}
\date{June 17, 2016}

\begin{abstract}
In a recent paper on exotic crossed products,
we included a lemma concerning ideals of the Fourier-Stieltjes algebra.
Buss, Echterhoff, and Willett have pointed out to us that our proof of this lemma contains an error.
In fact, it remains an open question whether the lemma is true as stated.
In this note we indicate how to contain the resulting damage.

Our investigation of the above question leads us to define two properties \emph{ordered} and \emph{weakly ordered} for invariant ideals of Fourier-Stieltjes algebras,
and we initiate a study of these properties.
\end{abstract}

\maketitle

\section{Introduction}\label{intro}

Throughout, let $G$ be a locally compact group.
In an effort to extend the class of groups for which the Baum-Connes conjecture is valid, Baum, Guentner, and Willett introduced in \cite{bgwexact} crossed-product functors, which transform actions of locally compact groups on $C^*$-algebras into $C^*$-algebras between the full and reduced crossed products.
Our approach to this has been to form crossed-product functors by applying coaction functors to the full crossed products.
This in particular requires us to study
exotic group $C^*$-algebras between $C^*(G)$ and $C^*_r(G)$ to form coaction functors.
When $G$ is discrete, Brown and Guentner introduced in \cite{bg} a certain method of generating exotic group $C^*$-algebras of $G$, starting with a $G$-invariant ideal $D$ of $\ell^\infty(G)$.
Their method carries over to locally compact $G$,
letting $D$ be a $G$-invariant ideal of either $L^\infty(G)$ or the algebra $C_b(G)$ of continuous bounded functions.
(This $L^\infty$-or-$C_b$ ambiguity is useful, because there are examples in which the ideal most naturally resides in one or the other.)
They used $D$ to define a class of unitary representations of $G$, and then applied standard $C^*$-representation theory to get an associated quotient of $C^*(G)$, denoted by $C^*_D(G)$, which was their main object of study.
Recently there has been a flurry of activity regarding
constructions using
these group $C^*$-algebras
(a brief sampling:
\cite{
brarua, 
bew2, 
bew, 
klqexact, 
klqfunctor, 
exotic,
Kyed,
ruawieexotic,
wiersmatensor, 
wiersmaexotic, 
wiersmalp}).

In \cite{graded} our strategy was to study the exotic group $C^*$-algebra $C^*_D(G)$ in terms of the dual space $C^*_D(G)^*$ of bounded linear functionals, which can be identified with a weak*-closed subspace of the Fourier-Stieltjes algebra $B(G)=C^*(G)^*$, namely the annihilator in $B(G)$ of the kernel of the quotient map $C^*(G)\to C^*_D(G)$.

However, recently a fundamental question has arisen:
the same $D$ can be used to arrive at a weak*-closed subspace of $B(G)$ in a different way: first form the intersection $E:=D\cap B(G)$, then take the weak*-closure $\bar E$ in $B(G)$.

\begin{q}\label{fundamental}
Does the above weak*-closure $\bar E$ coincide with the dual space $C^*_D(G)^*$?
\end{q}

In \cite[Lemma~3.5 (1)]{graded} we thought we had proven that the answer to \qref{fundamental} is yes. But it has recently been pointed out to us by Buss, Echterhoff, and Willett that our argument for one of the two inclusions is incorrect,
although we did give a correct proof of the inclusion $C^*_D(G)^*\subset \bar E$.
At this point we do not know whether \cite[Lemma~3.5 (1)]{graded} is true in general;
see \secref{investigate} for a discussion.
In this note we investigate this question, although we hasten to emphasize that we do not have a complete solution.

Our faulty proof of \cite[Lemma~3.5 (1)]{graded} seemed to depend upon the following property of a nonzero $G$-invariant ideal $E$ of $B(G)$:
$E=\spn\{E\cap P(G)\}$,
where $P(G)$ denotes the set of continuous positive definite functions
on $G$,
equivalently the set of positive linear functionals on $C^*(G)$.
We begin the study of this property,
which we call \emph{ordered},
in \secref{investigate},
and in \exref{counterex} we give a counterexample.
It transpires that to make things work we really only need $\spn\{E\cap P(G)\}$ to be weak* dense in $E$,
which we call \emph{weakly ordered},
and in \secref{investigate} we explore this property as well.
We do not know whether every nonzero $G$-invariant ideal of $B(G)$ has it.

In \secref{classical} we examine the above properties for certain ``classical'' ideals of $B(G)$,
particularly those arising from $L^p$ spaces,
and
in \secref{summary} we indicate how to ameliorate all damage caused by the suspect \cite[Lemma~3.5 (1)]{graded}.
The main take-away from all this is the following:
\textbf{\cite[Lemma~3.5 (1)]{graded} should have included the hypothesis that
the linear span of $E\cap P(G)$ is weak*-dense in
$E$.
To repair the damage,
any result that appeals to \cite[Lemma~3.5 (1)]{graded} should include this hypothesis.}

\section{The setup}\label{setup}

Since the $G$-invariant set $D$ discussed in \secref{intro} is only used to restrict the coefficient functions of representations, the intersection with the Fourier-Stieltjes algebra is all that matters.
So, let $E$ be a $G$-invariant (not necessarily weak*- or even norm-closed) vector subspace of $B(G)$.

\begin{rem}
If $E$ arises as in the introduction, by intersecting $B(G)$ with a $G$-invariant ideal of $L^\infty(G)$ or $C_b(G)$, then $E$ will also be an ideal of $B(G)$.
Although this property is in fact important to us for our main study of coaction functors, for the time being we only require $E$ to be a $G$-invariant subspace of $B(G)$.
\end{rem}

\begin{defn}
Let $U$ be a unitary representation of $G$ on a Hilbert space $H$, and let $\xi,\eta\in H$.
Define the \emph{coefficient function} $U_{\xi,\eta}$ by
\[
U_{\xi,\eta}(x)=\<U_x\xi,\eta\>\righttext{for}x\in G.
\]
If $\xi=\eta$ we write $U_\xi$.

We will find it convenient to adopt the convention that the zero representation of $G$ (on the 0-dimensional Hilbert space) is unitary.
\end{defn}

\begin{defn}[{see \cite[Definition~2.1]{bg}}]\label{E rep}
An \emph{$E$-representation} of $G$ is a triple $(U,H,H_0)$, where $U$ is a representation of $G$ on a Hilbert space $H$ and $H_0$ is a dense subspace of $H$ such that $U_{\xi,\eta}\in E$ for all $\xi,\eta\in H_0$.
(In \cite{bg}, the subspace $H_0$ is denoted by $\HH_0$.)

With our convention that the zero representation is unitary, we see that it is trivially an $E$-representation.
\end{defn}

The fussy notation $(U,H,H_0)$ will help us keep track of things; \cite{bg} just refers to $U$ itself as the $E$-representation\footnote{actually, a $D$-representation where $D$ is a $G$-invariant ideal of $\ell^\infty(G)$ for a discrete group $G$}, and sometimes we will also do this.

Of course, $U$ integrates to a nondegenerate representation of $C^*(G)$ on $H$, which we also denote by $U$.
When we refer to the kernel of $U$, we mean the ideal
$\ker U:=\{a\in C^*(G):U(a)=0\}$ of $C^*(G)$.
Part (1) of
the following definition is taken from 
\cite[Definition~2.2]{bg},
and part (2) from \cite[Definition~3.2]{graded}.

\begin{defn}\label{J def}
Let $E$ be a $G$-invariant subspace of $B(G)$.
\begin{enumerate}
\item
Define an ideal of $C^*(G)$ by
\[
J_E=\bigcap\{\ker U:\text{$U$ is an $E$-representation of $C^*(G)$}\},
\]
and then let
\[
C^*_{E,BG}(G)=C^*(G)/J_E.
\]

\item
On the other hand, by $G$-invariance the preannihilator $\pann E$ is also an ideal of $C^*(G)$,
so we can define another quotient $C^*$-algebra by
\[
C^*_{E,KLG}(G)=C^*(G)/\pann E.
\]
\end{enumerate}
\end{defn}

\begin{rem}
Since $E$ is a $G$-invariant subspace of $B(G)$,
the weak*-closure $\bar E$
is also $G$-invariant, and hence
is a $C^*(G)$-subbimodule of $B(G)$,
so by \cite[Corollary~3.10.8]{pedersen}
the preannihilator $\pann E=\pann \bar E$ is a closed ideal of $C^*(G)$.
This is also recorded in 
\cite[Lemma~2.10]{bew2}.
The notation $C^*_{E,BG}$ and $C^*_{E,KLQ}$ comes from \cite{bew}.
\end{rem}

The following is an alternative version of \qref{fundamental}, as we will see from the results in \secref{investigate}:

\begin{q}\label{JE}
If $E$ is a $G$-invariant subspace of $B(G)$, is
$J_E=\pann E$?
Equivalently, 
is
$C^*_{E,BG}(G)=C^*_{E,KLQ}(G)$?
\end{q}

For our purposes, we can assume that $E$ is actually an ideal of $B(G)$.
However, as far as we know \qref{JE} as stated remains open.

The inclusion $J_E\supset \pann E$ always holds,
as (correctly) shown in the second half of the proof of \cite[Lemma~3.5 (1)]{graded}, and here is the argument:
Let $a\in \pann E$, and let $(U,H,H_0)$ be an $E$-representation.
Then for all $\xi,\eta\in H_0$ we have $U_{\xi,\eta}\in E$, so
\[
0=U_{\xi,\eta}(a)=\<U(a)\xi,\eta\>,
\]
and so $U(a)=0$ by density. Thus $a\in J_E$.
Interestingly, this inclusion will also fall out of our investigation below (see \corref{one way}).

In \cite{graded} we gave an incorrect argument 
for
the opposite containment $J_E\subset \pann E$.

\begin{rem}
Just for fun, here is an alternative argument for the inclusion
proved above:
First, a set-theoretic technicality: there is a \emph{set}
$\RR$ of representations of $G$ such that
\[
J_E=\bigcap\{\ker U:U\in\RR\}.
\]
(The issue here is that in \defnref{J def} the intersection is indexed by a proper class, i.e., not a set. But we are intersecting a \emph{set} of ideals.)
\cite[Proposition~3.4.2 (i)]{dixmier} says that every state of $C^*(G)$ that vanishes on $J_E$ is a weak*-limit of states of the form
\[
\sum_{U\in\FF}U_{\xi_{U,H,H_0}},
\]
where $\FF\subset\RR$ is finite and $\xi_{U,H,H_0}\in H$ for all $(U,H,H_0)\in\FF$.
Now, by density each such state can be approximated in the weak*-topology by states of the same form but with $\xi_{U,H,H_0}\in H_0$ for each $(U,H,H_0)\in\FF$.
Thus every state in $J_E\ann$ is in $\bar E$.
Since $J_E\ann$ is the dual space of the quotient $C^*$-algebra $C^*(G)/J_E$,
every element of $J_E\ann$ is a linear combination of states in $J_E\ann$,
and hence $J_E\ann\subset \bar E$ by the preceding.
Therefore $J_E\supset \pann E$.
\end{rem}

\section{Investigation of the question}\label{investigate}

We now proceed to investigate \qref{JE}.
Throughout, $E$ will denote a $G$-invariant subspace of $B(G)$.
We want to find a reasonably general sufficient condition for $J_E=\pann E$.
In this section we illustrate one approach, mainly using the Hahn-Banach theorem.
First we introduce some auxiliary notation.
Recall from \defnref{E rep} that our notation for an $E$-representation is a triple $(U,H,H_0)$, where we keep track of the Hilbert space $H$ and the dense subspace $H_0$.

\begin{notn}
We write
\[
E_r:=\{U_{\xi,\eta}:(U,H,H_0)\text{ is an $E$-representation},\xi,\eta\in H_0\}.
\]
Note: as a consequence of our convention that the zero representation is (unitary and hence is) an $E$-representation, we see that $0\in E_r$.
\end{notn}

\begin{rem}\label{A_E}
Think of the elements of $E_r$ as ``representable'' (which is our motivation for the notation).
In \cite[Section~4]{wiersmalp}, Wiersma defines something similar but not quite the same --- he would write $A_E(G)$ for the set of all coefficient functions $U_{\xi,\eta}$, where now $\xi$ and $\eta$ are allowed to be any vectors from the Hilbert space $H$ of the $E$-representation $U$, not just from the dense subspace $H_0$.
\end{rem}

\begin{rem}
Somehow irritating, we still do not know whether $G$ has any nonzero $E$-representations (see \qref{exist} below).
\end{rem}

\begin{lem}\label{subspace}
For a $G$-invariant subspace of $B(G)$,
$E_r$
is a vector subspace of 
$E$.
\end{lem}

\begin{proof}
It is obvious that $E_r$ is closed under scalar multiplication.
Note that any direct sum of $E$-representations is an $E$-representation,
by the same reasoning as \cite[Remark~2.4]{bg}.
Let $f,g\in E$,
and choose $E$-representations $(U,H,H_0)$ and $(V,K,K_0)$ and vectors $\xi,\eta\in H_0$ and $\kappa,\zeta\in K_0$ such that
$f=U_{\xi,\eta}$ and $g=V_{\kappa,\zeta}$.
Then $(U\oplus V,H\oplus K,H_0\odot K_0)$ is an $E$-representation,
where $H_0\odot K_0$ stands for the algebraic direct sum of the vector subspaces $H_0$ and $K_0$,
and
\[
U_{\xi,\eta}+V_{\kappa,\zeta}=(U\oplus V)_{(\xi,\kappa),(\eta,\zeta)}.
\qedhere
\]
\end{proof}

\begin{defn}\label{pg property}
For a (not necessarily $G$-invariant) subspace $E$ of $B(G)$, put
$E_0=\spn\{E\cap P(G)\}$.
We say that $E$ is \emph{ordered} if $E_0=E$,
and we say that $E$ is \emph{weakly ordered} if $E_0$ is weak* dense in $E$.
\end{defn}

Although in the above definition we temporarily removed the tacit assumption that $E$ is $G$-invariant, we will continue to impose this assumption unless otherwise specified.

We will see that not every subspace of $B(G)$ is ordered, and we begin with an obvious obstruction:
First recall that the involution in the Fourier-Stieltjes algebra $B(G)$ is given by
\[
\wilde f(x)=\bar{f(x\inv)}.
\]
It follows from the properties of duals of $C^*$-algebras that
every ordered subspace $E$ of $B(G)$
is self-adjoint: if $f\in E$ then also $\wilde f\in E$.

\begin{rem}
The above terminology ``ordered'' makes sense because it is precisely what it means for the self-adjoint part of $E$ to be a partially ordered (real) subspace of the self-adjoint part of $B(G)$.
It is slightly less obvious that the terminology ``weakly ordered'' is sensible, but it is somehow related to the property ``ordered'' and is obviously weaker (and as we will show, strictly so).
\end{rem}

We will show in \exref{counterex} that in fact not every $G$-invariant ideal of $B(G)$ is ordered.

\begin{q}\label{PG q}
Is every $G$-invariant subspace of $B(G)$ weakly ordered?
\exref{jack} below gives some negative evidence, although it does not furnish a counterexample.
\end{q}

\begin{lem}
$E$ is weakly ordered if and only if $\pann E=\pann E_0$.
\end{lem}

\begin{proof}
Since $E_0\subset E$, this follows from the Hahn-Banach theorem.
\end{proof}

The following is equivalent to \cite[Lemma~2.15]{bew2}
(see also \cite[Proposition~4.3]{wiersmalp}),
with a somewhat different proof.

\begin{prop}[\cite{bew2}]\label{closed}
If $E$ is closed in the norm of $B(G)$,
then it is ordered.
\end{prop}

\begin{proof}
By \cite[Theorem~3.17]{arsac} $E$ is
the set of all coefficient functions
of some representation $U$ of $G$, and then by 
\cite[Proposition~2.2]{arsac} 
$E$ is the predual of the von Neumann algebra generated by $U(G)$.
It then follows (see, for example, 
\cite[Proposition~3.6.2]{pedersen}
or
\cite[Theorem~12.3.3]{dixmier})
that $E$ is the linear span of positive linear functionals.
\end{proof}

\begin{rem}
\propref{closed} obviously applies in particular to 
situations where 
$E$ is relatively closed in $B(G)$ as a subset of $C_b(G)$ with the sup norm,
most importantly
$E=C_0(G)\cap B(G)$,
as observed in
\cite[Lemma~2.15]{bew2}.
\end{rem}

\begin{rem}
Here is an alternative, somewhat more elementary, argument:
Since $E$ is a subspace, trivially $E_0\subset E$.
For the opposite containment, let $f\in E$.
To show that $f\in E_0$, without loss of generality $f\ne 0$.
By \cite[Proposition~4.1]{wiersmalp} (for example), there is 
a representation
$U$ and vectors $\xi,\eta$ such that $f=U_{\xi,\eta}$ and $\|f\|=\|\xi\|\|\eta\|$ (where the norm of $f$ is taken in $B(G)$).

Let
\begin{align*}
H_\xi&=\clspn\{U_x\xi:x\in G\}
\\
H_\eta&=\clspn\{U_x\eta:x\in G\},
\end{align*}
and let $P_\xi$ and $P_\eta$ be the orthogonal projections onto these respective subspaces.
Since $P_\eta$ commutes with $U$,
\[
f(x)=\<U_x\xi,P_\eta\eta\>=\<U_xP_\eta\xi,\eta\>.
\]
Thus by construction we have
\[
\|P_\eta\xi\|\|\|\eta\|\ge \|f\|=\|\xi\|\|\eta\|,
\]
so $\|P_\eta\xi\|\ge \|\xi\|$,
and 
hence, since $P_\eta$ is a projection,
we must have
$P_\eta\xi=\xi$, giving $\xi\in H_\eta$.
Similarly $\eta\in H_\xi$.
Thus in fact $H_\xi=H_\eta$.
Since $E$ is $G$-invariant and norm-closed,
the coefficient functions $U_{\xi',\eta'}$
are in $E$ for all 
$\xi',\eta'\in H_\xi$.
Thus
\begin{align*}
f
&=U_{\xi,\eta}
=\frac 14\sum_{k=0}^3i^kU_{\xi+i^k\eta}\in E_0.
\qedhere
\end{align*}
This argument should be compared to \cite[proof of Proposition~4.3]{wiersmalp} and \cite[proof of Lemma~2.15]{bew2}.
\end{rem}

\begin{prop}\label{represent}
$E_0=E_r$.
\end{prop}

\begin{proof}
Let $f\in E\cap P(G)$.
Choose a cyclic representation $U$ of $C^*(G)$ on a Hilbert space $H$, with cyclic vector $\xi$, such that $f=U_\xi$.
Let
\[
H_0=\spn_{x\in G}\{U_x\xi\},
\]
which is a dense subspace of $H$.
For all $a\in C^*(G)$ and $x,y\in G$,
\[
\bigl\<U(a)U_x\xi,U_y\xi\bigr\>=x\cdot U_\xi\cdot y\inv(a),
\]
and $x\cdot U_\xi\cdot y\inv\in E$,
so $(U,H,H_0)$ is an $E$-representation.
Therefore $f\in E_r$.
By \lemref{subspace} it follows that $E_0\subset E_r$.

On the other hand, if $(U,H,H_0)$ is an $E$-representation and $\xi,\eta\in H_0$,
then
\[
U_{\xi,\eta}=\frac 14\sum_0^3i^kU_{\xi+i^k\eta},
\]
and $\xi+i^k\eta\in H_0$ for $k=0,\dots,3$.
Thus $E_r\subset E_0$.
\end{proof}

\begin{rem}
Recall from \remref{A_E} that in \cite{wiersmalp} Wiersma writes $A_E(G)$ for the set of all coefficient functions of $E$-representations.
Using our notation, \cite[Proposition~4.3]{wiersmalp} says that $A_E(G)=\bar{E_0}$.
Thus by \propref{represent}, $A_E(G)=\bar{E_r}$.
\end{rem}

\begin{cor}\label{JE E0}
$J_E=\pann E_0$.
\end{cor}

\begin{proof}
We have
\begin{align*}
J_E
&=\bigcap\{\ker U:(U,H,H_0)\text{ is an $E$-representation}\}
\\&=\bigcap\{\ker U_{\xi,\eta}:(U,H,H_0)\text{ is an $E$-representation},\xi,\eta\in H_0\}
\\&=\pann E_r
\\&=\pann E_0,
\end{align*}
where we used density in the second step.
\end{proof}

\begin{rem}\label{one way}
We can use the above results to give an alternative proof that
$\pann E\subset J_E$:
Since $E_0\subset E$, we have
$\pann E\subset \pann E_0$,
so the inclusion follows from \corref{JE E0}.
\end{rem}

The following corollary is essentially the second half of \cite[Proposition~2.13]{bew2}.

\begin{cor}[\cite{bew2}]\label{suffice}
If $E$ is ordered then $J_E=\pann E$.
\end{cor}

\begin{proof}
Since we already know that $\pann E\subset J_E$,
it suffices to show that $J_E\subset \pann E$.
Let $a\in J_E$ and $f\in E$.
By assumption we can choose an $E$-representation $(U,H,H_0)$ and $\xi,\eta\in H_0$ such that $f=U_{\xi,\eta}$.
Since $a\in \ker U$, we have
\[
0=\<U(a)\xi,\eta\>=U_{\xi,\eta}(a)=f(a).
\]
Therefore $a\in \pann E$.
\end{proof}

We now recover \cite[Corollary~2.14]{bew2} (with essentially the same proof), which perfects \corref{suffice}:

\begin{cor}[\cite{bew2}]\label{equivalent}
$J_E=\pann E$ if and only if $E$ is weakly ordered.
\end{cor}

\begin{proof}
By \corref{JE E0}, $J_E=\pann E$ if and only if $\pann E=\pann E_0$,
which in turn is equivalent to $\bar E=\bar{E_0}$ (where the bars denote weak*-closures), and the result follows since $E_0\subset E$.
\end{proof}

\begin{rem}\label{contain ag}
\cite[Lemma~3.14]{graded} (which is correctly proved)
says
that if $E$ is a nonzero $G$-invariant ideal of $B(G)$
then $E$ contains a dense ideal of $A(G)$ contained in $C_c(G)$,
and consequently
then the norm closure of $E$ contains $A(G)$,
and the weak* closure contains $B_r(G)$.\footnote{It might be worthwhile mentioning that there was a slight redundancy in our argument: it was not necessary to ensure that the ideal separates points in $G$.}
Thus we always have $\pann E\subset \ker \lambda$.
Suppose that $E_0\ne \{0\}$.
Then $E_0$ is also a nonzero $G$-invariant ideal of $B(G)$,
so by the same argument
it follows that the (perhaps) smaller ideal $E_0$ is still weak* dense in $B_r(G)$, and so
\[
\pann E\subset \pann E_0=J_E\subset \pann B_r(G)=\ker \lambda.
\]
Thus
if $E_0\ne \{0\}$ and
$\pann E=\ker\lambda$ then $E$ is weakly ordered.
\end{rem}

\begin{q}\label{exist}
If $E$ is a nonzero $G$-invariant ideal of $B(G)$,
does $G$ have a nonzero $E$-representation?
Equivalently, is the subspace $E_r$ of $E$ nontrivial?
Bizarrely, this question is apparently open for general locally compact groups.
\end{q}

For $G$ discrete, the answer is yes:

\begin{prop}
If $G$ is discrete, then $\lambda$ is an $E$-representation.
\end{prop}

\begin{rem}
As (essentially) mentioned in \cite{bg}, if $G$ is discrete then $E\supset c_c(G)$, and it follows that $\lambda$ is an $E$-representation.
\end{rem}

\begin{rem}
Trivially, if $E$ is weakly ordered, then $G$ has a nonzero $E$-representation --- and conversely if $G$ is amenable (see \propref{amenable} below), although in some sense the amenable case is of no interest with regard to these matters.
Also, if $F$ is a nonzero $G$-invariant ideal of $B(G)$ containing $E$, and if $G$ has a nonzero $E$-representation, then trivially it has a nonzero $F$-representation.
\end{rem}

\begin{prop}\label{amenable}
Let $G$ be amenable, and let $E$ be a nonzero $G$-invariant ideal of $B(G)$.
If $G$ has a nonzero $E$-representation, then $E$ is weakly ordered.
\end{prop}

\begin{proof}
By hypothesis, $E_0\ne \{0\}$.
As we point out in \remref{contain ag}, $\pann E\subset \ker\lambda$. Since $G$ is amenable, $\ker\lambda=\{0\}$. Consequently, $\pann E=\ker\lambda$.
Consulting \remref{contain ag} again we conclude that
$E$ is weakly ordered.
\end{proof}

\begin{ex}\label{jack}
Taking $G$ to be the circle group $\T$,
we will give an example of a weak*-dense 
subspace $E$ for which $E_0=\{0\}$,
and so 
$E$ fails to be weakly ordered
in a very strong way.
Note that
our example is 
neither $G$-invariant nor an ideal of $B(G)$.
Let
\[
E=\spn\{z^{n+1}-z^n:n\in\N\}\subset B(\T).
\]
Then $E$ is a 
subspace of $B(\T)$,
and we claim that $E\cap P(\T)=\{0\}$.
By Bochner's theorem, it suffices to show that the Fourier transform
\[
\what E=\{\what f:f\in E\}=\spn\{\delta_{n+1}-\delta_n:n\in\N\}
\]
contains no nonzero positive measure on $\Z$.
Let
\[
\mu=\sum_{n=-k}^kc_n(\delta_{n+1}-\delta_n),
\]
and assume that $\mu$ is positive and nonzero.
Clearly $k>0$, and without loss of generality $c_k\ne 0$.
For each $n\in\Z$ let $p_n=\Chi_{\{n\}}\in c_0(\Z)$.
Then
\[
0\le \<p_{k+1},f\>=c_k,
\]
so $c_k>0$.
Next,
\[
0\le \<p_k,f\>=c_{k-1}-c_k,
\]
so $c_{k-1}\ge c_k$.
Continuing in this way, we find
\[
c_{-k}\ge c_{-k+1}\ge \cdots \ge c_k>0.
\]
But
\[
0\le \<p_{-k},f\>=-c_{-k},
\]
giving a contradiction.
Note that $\what E$ is weak*-dense in the space $\ell^1(\Z)$ of complex measures on $\Z$,
and so $E$ is weak*-dense in $B(\T)$.
We thank J.~Spielberg for fruitful discussion that led to this example.
\end{ex}

\begin{rem}
Let $E$ be a nonzero $G$-invariant ideal of $B(G)$.
Suppose that there exists at least one nonzero $E$-representation $U$ of $G$, equivalently $E_0\ne \{0\}$.
If $V$ is any representation of $G$, then $U\otimes V$ is an $E$-representation since $E$ is an ideal.
In particular, $U\otimes \lambda$ is an $E$-representation.
By Fell's trick, $U\otimes \lambda$ is unitarily equivalent to a multiple of $\lambda$.
Thus this multiple of $\lambda$ is an $E$-representation.
We would like to conclude that $\lambda$ is an $E$-representation,
but this seems to require that the class of $E$-representations be closed under taking subrepresentations.
It is not clear to us why this would be true.
But if it were then we could conclude that
$E_0$ contains every convolution $\xi*\eta$ for $\xi,\eta\in L^2(G)$.
Note that our assumptions imply that $E_0$ is also a nonzero $G$-invariant ideal of $B(G)$,
so by \remref{contain ag} we already knew --- for different reasons --- that the norm closure of $E_0$ contains $A(G)$.
\end{rem}

\section{The classical ideals}\label{classical}

Inspired by the work of Brown and Guentner \cite{bg},
the main examples of $G$-invariant subspaces $E\subset B(G)$ that we want to include are actually ideals:
\begin{enumerate}
\item $C_c(G)\cap B(G)$;
\item $C_0(G)\cap B(G)$;
\item $L^p(G)\cap B(G)$.
\end{enumerate}

\begin{prop}
The ideals \textup{(1)} and \textup{(2)} are ordered.
\end{prop}

\begin{proof}
The first case is very well-known: by \cite[Proposition~3.4]{eymard},
\begin{align*}
C_c(G)\cap B(G)
&=\spn\{C_c(G)\cap P(G)\}
\\&=\spn\{C_c(G)\cap B(G)\cap P(G)\}.
\end{align*}
The second case is immediate from \propref{closed}.
\end{proof}

However, for ideals of type (3) things are murky.
We do not even know whether
they are all weakly ordered.
Since this is an important source of ideals of $B(G)$,
we examine this more closely.
For $1\le p\le \infty$ let
\[
E^p=L^p(G)\cap B(G).
\]
In particular, $E^\infty=B(G)$.

Note that if $G$ is unimodular then every $E^p$ is at least self-adjoint in $B(G)$.
However, this does not hold generally, as the following example shows.

\begin{ex}\label{counterex}
Here we
show that there are groups for which the ideal $E^1$ is not ordered,
by showing that it is not self-adjoint in $B(G)$.
It seems plausible, but not clear, that this carries over to arbitrary $p<\infty$ by embellishing the computations.

We want to find $f\in E^1$ such that $\wilde f\notin E^1$.
Let $g,h\in L^1(G)\cap L^2(G)$ be nonnegative, and put
\[
f(x)=\<\lambda_x g,h\>=\int f(x\inv y)g(y)\,dy.
\]
Then $f\in B(G)$, and
\begin{align*}
\|f\|_1
&=\int \int g(x\inv y)h(y)\,dy\,dx
\\&=\int \int g(x\inv)\,dx\,h(y)\,dy\righttext{(after $x\mapsto yx$)}
\\&=\|g \Delta\inv\|_1\|h\|_1,
\end{align*}
where here we write $\Delta\inv$ for the reciprocal $1/\Delta$.

On the other hand, $\wilde f\in B(G)$, and
\begin{align*}
\wilde f(x)
&=f(x\inv)
=\int g(xy)h(y)\,dy,
\end{align*}
so
\begin{align*}
\|\wilde f\|_1
&=\int \int g(xy)\,dx\,h(y)\,dy
\\&=\int \int g(x)\,dx\,\Delta(y)\inv h(y)\,dy
\\&=\|g\|_1\|\Delta\inv h\|_1.
\end{align*}
Now, we impose further assumptions on $g,h$:
\begin{align*}
&\|g\|_1=\infty
\\
&\|g\Delta\inv\|_1<\infty
\\
&0<\|h\|_1,
\|\Delta\inv h\|_1<\infty.
\end{align*}
Then $f\in E^1$ and $\wilde f\notin E^1$,
so $E^1$ is not ordered.

We can easily choose a suitable $h$ --- for instance, let $h\in C_c(G)$ be nonnegative and not identically 0.
It seems likely that we can also choose a suitable $g$ in any nonunimodular group.
For a specific example,
let $G$ be the $ax+b$ group $\R^+\times\R$, with operation
\[
(x,y)(u,v)=(xu,xv+y).
\]
Recall that the Haar measure and modular function are given by
\begin{align*}
d(x,y)&=\frac{dx\,dy}{x^2}
\\
\Delta(x,y)&=\frac 1x.
\end{align*}
We look for $g$ of the form
\[
g(x,y)=\phi(x)\psi(y),
\]
with $\phi,\psi\ge 0$.
We need $g\in L^2$,
which means
\[
\infty>\int_G g(x,y)^2\,d(x,y)=\int_0^\infty \frac{\phi(x)^2}{x^2}\,dx\int_\R \psi(y)^2\,dy.
\]
We also need $g\Delta\inv$ integrable but $g$ nonintegrable,
which means
\[
\infty=\int_G g(x,y)\,d(x,y)=\int_0^\infty \frac{\phi(x)}{x^2}\,dx\int_\R \psi(y)\,dy
\]
and
\[
\infty>\int_G \frac{g(x,y)}{\Delta(x,y)}\,d(x,y)=\int_0^\infty \frac{\phi(x)}{x}\,dx\int_\R \psi(y)\,dy.
\]
These conditions are all met with, e.g.,
\begin{align*}
\phi(x)&=xe^{-x}
\\
\psi(y)&=e^{-y^2}.
\end{align*}

\end{ex}

\begin{q}\label{PG prop}
When
is $E^p$
\begin{enumerate}
\item
ordered?

\item
weakly ordered?
\end{enumerate}
\end{q}

Trivially:

\begin{prop}
$E^\infty$ is ordered.
\end{prop}

\begin{prop}\label{E2 weak}
If $1\le p\le 2$ then $E^p$ is weakly ordered.
\end{prop}

\begin{proof}
\cite[Proposition~4.2]{graded} says that
$\pann E^p=\ker\lambda$,
and the result follows.
\end{proof}

\begin{cor}\label{weaker}
For some groups, there are $G$-invariant ideals of $B(G)$ that are weakly ordered but not ordered.
\end{cor}

We can at least answer \qref{exist} affirmatively for $E^p$:

\begin{cor}
For every $p$, there is a nonzero $E^p$-representation of $G$.
\end{cor}

\begin{proof}
By \propref{E2 weak}, $E^p$ is weakly ordered for all $p\le 2$,
and so has a nonzero $E^p$-representation,
and hence has a nonzero $E^p$ representation for all $p>2$ as well,
because if $p>q$ then $E^p\supset E^q$,
since $B(G)$ consists of bounded functions.
\end{proof}

\begin{rem}
\propref{E2 weak} is also implied by \cite[Proposition~4.4 (i)]{wiersmalp}, which says that $A_{L^p}(G)=A(G)$ for all $p\in [1,2]$.
\end{rem}

\begin{prop}\label{hr}
If $G$ is abelian, then
$E^2$ is ordered.
\end{prop}

\begin{proof}
The Fourier transform takes $L^2(\what G)\cap L^1(\what G)$ bijectively onto $L^2(G)\cap A(G)$.
Now, $L^2(\what G)\cap L^1(\what G)$ is the linear span of the nonnegative functions it contains,
so $L^2(G)\cap A(G)$ is the linear span of the positive definite functions it contains.
The result now follows from \propref{magnus} below.
\end{proof}

In the above proof we appealed to the following elementary fact, which is perhaps folklore:

\begin{prop}\label{magnus}
For any locally compact group $G$, if $1\le p\le 2$ then
$L^p(G)\cap B(G)=L^p(G)\cap A(G)$.
\end{prop}

\begin{proof}
It suffices to show that if $f\in L^p(G)\cap B(G)$ then $f\in A(G)$.
Since $B(G)$ consists of bounded functions, we have
\[
\bigl(L^p(G)\cap B(G)\bigr)\subset \bigl(L^2(G)\cap B(G)\bigr)
\]
so it suffices to prove the result for the special case $p=2$.
Choose a representation $U$ of $G$ and vectors $\xi,\eta$ such that $f=U_{\xi,\eta}$.
For any $g\in L^1(G)\cap L^2(G)$,
define $\psi_g\in A(G)$ by
\begin{align*}
\psi_g(x)
&=\<\lambda_xg,\bar f\>
\\&=\int g(x\inv y)f(y)\,dy
\\&=\int g(y)\<U_{xy}\xi,\eta\>\,dy
\\&=\int \<U_xU_yg(y)\xi,\eta\>\,dy
\\&=\<U_xU_g\xi,\eta\>.
\end{align*}
It follows that for any $a\in C^*(G)$ we have
\[
\psi_g(a)
=\<U_aU_g\xi,\eta\>.
\]
Letting $U_g\xi\to \xi$ in the norm of the Hilbert space of $U$, we have $\psi_g\to f$ in the $B(G)$-norm,
and therefore $f\in A(G)$.
\end{proof}

\begin{rem}
We conjecture that the conclusion of \propref{hr} holds for all unimodular groups.
\end{rem}

\begin{rem}\label{no magnus p>2}
Here we show that \propref{magnus} does not extend to $p>2$:
for $G=SL(2,\R)$,
by \cite[Theorem~7.2]{wiersmalp}
\[
\bar{L^p(G)\cap B(G)}\wkstcl\ne \bar{L^2(G)\cap B(G)}\wkstcl,
\]
whereas
\[
\bar{L^p(G)\cap A(G)}\wkstcl=B_r(G)=\bar{L^2(G)\cap B(G)}\wkstcl.
\]
\end{rem}

\begin{rem}
In view of the discussion in this section, it might be worthwhile to consider three possible (re-)definitions of the $G$-invariant ideal $E^p$ of $B(G)$:
\begin{enumerate}
\item $L^p(G)\cap B(G)$;

\item $\{f\in B(G):f,\wilde f\in L^p(G)\}$;

\item $\spn\{L^p(G)\cap P(G)\}$.
\end{enumerate}

(1) is of course how we defined the notation $E^p$ in this paper,
and (2) is the self-adjoint part of (1).
Since (3) is always self-adjoint, we obviously have (1) $\supset$ (2) $\supset$ (3).

(3) is the convention used in \cite{bew}, with good reason.

By our definition, (1) = (3) only when (1) is ordered, which we have seen does not always occur;
for example, it happens for $p=2$ and $G$ unimodular (in which case in fact (1) = (2) = (3)),
but (1) is not ordered for some (all?) nonunimodular $G$.

If $G$ is unimodular then (1) = (2).
\end{rem}

\begin{rem}
Here is a frustrating illustration of our ignorance:
First recall that \cite[Theorems~7.2 and 7.3]{wiersmalp} show that for $G=SL(2,\R)$
the large ideals of $B(G)$
consist precisely of $\bar{E^p}^{\text{weak*}}$ for $1\le p\le \infty$, and moreover for $2\le p\le \infty$ these ideals are all distinct, the extremes being $B_r(G)$ for $p=2$ and $B(G)$ for $p=\infty$.
Now let $E=E^p$ for some $p\in (2,\infty)$.
Then $\bar E_0^{\text{weak*}}$ is a large ideal of $B(G)$,
and so by Wiersma's results
it coincides with $\bar{E^{p'}}^{\text{weak*}}$ for
a unique 
$p'\in [2,p)$.
But since we don't know whether $E^p$ is weakly ordered,
we can't determine whether $p'=p$.
\end{rem}

\section{Summary}\label{summary}

In this section we 
summarize the preceding discussion regarding the offending
\cite[Lemma~3.5 (1)]{graded}.

Most importantly,
in all cases where \cite[Lemma~3.5 (1)]{graded} is used, the hypothesis that $E_0$ be weak*-dense in $E$ should be mentioned and verified.
We emphasize that for large ideals $E$ of $B(G)$
(or even just nonzero $G$-invariant norm-closed ideals),
there is no problem: $C^*_{E,BG}(G)=C^*_{E,KLQ}(G)$.
However, if $E$ is just a nonzero $G$-invariant ideal, then it's probably best to use the Brown-Guentner convention for $C^*_E(G)$, namely take
\[
C^*_E(G)=C^*_{E,BG}(G)=C^*(G)/J_E
\]
rather than
$C^*_{E,KLQ}(G)=C^*(G)/\pann E$.
As we have seen, if we replace the given $E$ by $E_0:=\spn(E\cap P(G))$ then the two approaches give the same group $C^*$-algebra.
Note that this is the approach of \cite[Example~2.16]{bew2} for $E^p$.

We give a few examples of how results that mention \cite[Lemma~3.5 (1)]{graded} should be adjusted.
In that lemma itself, also item (2) depends upon the new hypothesis,
since part of \cite[Lemma~3.5 (2)]{graded} is equivalent to the equality $C^*(G)/\pann E=C^*(G)/J_E$.

\cite[Corollary~3.6 (1)]{graded}
says that a representation $U$ of $G$ is an $E$-representation if and only if $\ker U\supset \pann E$. 
Since this depends upon $\pann E=J_E$, the new hypothesis $\bar{E_0}=\bar E$ should be added here.

Similarly, the new hypothesis should be added to \cite[Observation~3.8 and Remark~3.18]{graded}.

\cite[Section~4]{graded}
is explicitly about the classical ideals mentioned in \secref{classical}, and in particular the problem arises in discussions of $C^*_{L^p(G)}(G)$.
However, it follows from  \propref{E2 weak} that 
\cite[Proposition~4.2]{graded}
is correct as stated.

\begin{rem}
This might be a convenient place to correct another (relatively harmless) misstatement in \cite[Remark~4.3]{graded}, where it is asserted that,
for discrete $G$,
the weak*-closure of $C_0(G)\cap B(G)$ being strictly larger than $B_r(G)$ occurs precisely when $G$ is a-T-menable but nonamenable,
and that for perhaps the earliest result along these lines one can see \cite{menchoff}.
This is garbled in a couple of ways.
First of all, Menchoff's 1916 paper gives examples of singular measures whose Fourier coefficients tend to zero, thus showing that
the intersection $E:=C_0(G)\cap B(G)$ can properly contain the Fourier algebra $A(G)$, even for $G=\T$.
This certainly does not, however, illustrate the phenomenon of the weak*-closure $\bar E$ being strictly larger than $B_r(G)$.
The second blunder here is the use of the word ``precisely'';
a-T-menability is equivalent to
$J_E=0$, and hence to $J_E\ann=B(G)$.
This property
certainly implies $\bar E=B(G)$, which is strictly larger than $B_r(G)$ if $G$ is nonamenable --- 
so, 
when $G$ is a-T-menable but nonamenable 
we have $\bar E\supsetneq B_r(G)$ for $E=C_0(G)\cap B(G)$.
On the other hand, it is not clear to us that $\bar E\supsetneq B_r(G)$ implies a-T-menability.

While we are at it, we can mention one more minor slip in \cite{graded}: in the bibliographic entry for \cite{menchoff} the French word ``d\'eveloppement'' is misspelled.
\end{rem}

%----------------------------

%\bibliographystyle{amsplain}
%\bibliography{cstar}

\providecommand{\bysame}{\leavevmode\hbox to3em{\hrulefill}\thinspace}
\providecommand{\MR}{\relax\ifhmode\unskip\space\fi MR }
% \MRhref is called by the amsart/book/proc definition of \MR.
\providecommand{\MRhref}[2]{%
  \href{http://www.ams.org/mathscinet-getitem?mr=#1}{#2}
}
\providecommand{\href}[2]{#2}

\end{document}